\documentclass[11pt,letterpaper,reqno]{amsart}

\usepackage{amsmath,amssymb,amsthm,amsfonts}
\usepackage{bbm}
\usepackage{enumitem}
\usepackage{booktabs}
\usepackage{graphicx}
\usepackage[T1]{fontenc}
\usepackage{doi}
\usepackage{float}

\usepackage{tikz}
\usetikzlibrary{positioning, shapes.geometric, arrows.meta}
\usepackage{pgfplots}
\pgfplotsset{compat=1.18}

\usepackage{bookmark}
\usepackage{hyperref}
\hypersetup{pdfstartview={FitH}}

\newtheorem{thm}{Theorem}[section]
\newtheorem{lem}[thm]{Lemma}
\newtheorem{prop}[thm]{Proposition}

\newtheorem{conj}[thm]{Conjecture}
\theoremstyle{definition}

\newtheorem{remark}[thm]{Remark}

\numberwithin{equation}{section}

\DeclareMathOperator{\adj}{adj}
\DeclareMathOperator{\diag}{diag}
\DeclareMathOperator{\Ree}{Re}

\newcommand{\J}{J}
\newcommand{\one}{\mathbf{1}}

\newcommand{\mnr}{M_n(\mathbb{R})}
\newcommand{\R}{\mathbb{R}}
\newcommand{\C}{\mathbb{C}}
\newcommand{\rank}{\operatorname*{rank}}
\begin{document}
	
	\title[Johnson's determinantal identity]{
	Johnson's determinantal identity for contiguous minors of Toeplitz matrices, with an accretive extension}
	
	\author[T.~Zhang]{Teng Zhang}
	\address{School of Mathematics and Statistics, Xi'an Jiaotong University, Xi'an 710049, P. R. China}
	\email{teng.zhang@stu.xjtu.edu.cn}
	
	
	\begin{abstract}
		Let $A$ be an $n\times n$ real Toeplitz matrix satisfying
		$A+A^{\top}=2\J_n$, where $\J_n$ is the all-ones matrix.
		If $A_r(i,j)$ denotes the $r\times r$ contiguous submatrix of $A$ consisting of rows
		$i,i+1,\dots,i+r-1$ and columns $j,j+1,\dots,j+r-1$, then for every $n\ge 2$ one has
		\[
		\det A_{n-1}(1,2)+\det A_{n-1}(2,1)=2\det A_{n-1}(1,1).
		\]
		This confirms a conjecture of Charles R.~Johnson (2003). The proof combines a rank-one determinant expansion with Dodgson's condensation formula, and then invokes a polynomial-identity argument in the Toeplitz parameters: after obtaining an equality of
		squares in the integral domain $\mathbb{Z}[b_1,\dots,b_{n-1}]$, we factor it to deduce an identity
		up to sign and determine the sign by a suitable specialization.
	 We also give an extension of the Bayat--Teimoori arithmetic--geometric mean identity:
		for every real accretive matrix $A$, one has the sharp inequality
		\[
		\sqrt{\det A_{n-1}(1,1)\ \det A_{n-1}(2,2)}
		\ \ge\
		\left|\frac{\det A_{n-1}(1,2)+\det A_{n-1}(2,1)}{2}\right|,
		\]
	with equality whenever the symmetric part has rank one, i.e.\ $A+A^{\top}=\alpha\,ww^{\top}$
	for some $\alpha\in\mathbb R$ and $w\in\mathbb R^n\setminus\{0\}$,
	recovering the Bayat--Teimoori equality as a special case.
	\end{abstract}
	
	\maketitle
	
	\section{Introduction}
	
Let $\mnr$ be the set of all $n\times n$ real matrices.
	For $A\in \mnr$, integers $1\le r\le n$, and $1\le i,j\le n-r+1$, we write
	\[
	A_r(i,j):=\bigl(a_{pq}\bigr)_{\substack{i\le p\le i+r-1\\ j\le q\le j+r-1}}
	\]
	for the $r\times r$ contiguous submatrix of $A$ starting at row $i$ and column $j$. We denote by $A^\top$ the transpose of $A$. For two Hermitian matrices $X, Y$, we say $X < (\le)\, Y$ if $X-Y$ is positive (semi)definite.
	
	A matrix $A=(a_{ij})_{1\le i,j\le n}$ is \emph{Toeplitz} if $a_{ij}$ depends only on $j-i$ (i.e.\ entries
	are constant along diagonals). Let $\J_n$ be the $n\times n$ matrix whose entries are all $1$.
	
	At the combinatorial workshop on combinatorics, linear algebra, and graph theory held in Tehran in 2003,
	Charles R.~Johnson posed to M.~Bayat and H.~Teimoori a problem (see the statement in \cite{BT11})
	concerning the evaluation of determinants via Dodgson's condensation formula \cite{Dod66,Zei97}
(also known as  the Lewis Carroll identity or the Desnanot--Jacobi identity). Most proofs of Dodgson's condensation formula are algebraic; the first complete proof for arbitrary \(n\) was given by Jacobi in 1833~\cite{Mu06}. Zeilberger provided a combinatorial proof in 1997~\cite{Zei97}. Dodgson's condensation formula plays a central role in Dodgson condensation, an iterative technique for evaluating determinants, see \cite{Mui60}.
	
	\begin{thm}[Desnanot--Jacobi (Dodgson condensation)]\label{thm:dodgson}
		Let $n\ge 3$ and $A\in \mnr$. Then
		\begin{equation}\label{eq:dodgson}
			(\det A)\,\det A_{n-2}(2,2)
			=
			\det\!\left[
			\begin{array}{cc}
				\det A_{n-1}(1,1) & \det A_{n-1}(1,2) \\
				\det A_{n-1}(2,1) & \det A_{n-1}(2,2)
			\end{array}
			\right].
		\end{equation}
	\end{thm}
	Motivated by this framework, Johnson focused on a structured class of matrices--Toeplitz matrices, and conjectured a linear relation among \((n-1)\times(n-1)\) contiguous minors.

	\begin{conj}[Johnson]\label{conj:johnson}
		Let $n\ge 2$ and $A\in \mnr$ be a Toeplitz matrix satisfying	$A+A^{\top}=2\J_n$. Then
		\begin{equation}\label{eq:johnson}
			\det A_{n-1}(1,2)+\det A_{n-1}(2,1)=2\det A_{n-1}(1,1).
		\end{equation}
	\end{conj}
	
	In this paper, we confirm Conjecture \ref{conj:johnson}.
	
	A closely related result of Bayat--Teimoori \cite[Theorem 1.3]{BT11} is an arithmetic--geometric mean identity for the
	same minors under the rank-one symmetric-part hypothesis.
	
	\begin{thm}[Bayat--Teimoori]\label{thm:BT}
		Let $n\ge 2$ and $A\in \mnr$. If $A + A^{\top} = \alpha \J_n$ for some $\alpha\in \R$, then
		\begin{equation*}
			\sqrt{\det A_{n-1}(1,1)\;\det A_{n-1}(2,2)}
			=
			\left|\frac{\det A_{n-1}(1,2)+\det A_{n-1}(2,1)}{2}\right|.
		\end{equation*}
	\end{thm} 
	We present the following rank-one symmetric part generalization of Theorem \ref{thm:BT}.
		\begin{thm}\label{thm:rankone}
		Let $A\in\mnr$ satisfy
$	A+A^{\top}=\alpha\,ww^{\top},$ where $
 \alpha\in\R, w\in\R^n \backslash \{0\}$. Then 
		\[
		\sqrt{\det A_{n-1}(1,1)\ \det A_{n-1}(2,2)}
		=
		\left|\frac{\det A_{n-1}(1,2)+\det A_{n-1}(2,1)}{2}\right|.
		\]
		In particular, taking $w=(1,\dots,1)^{\top}$ recovers Theorem~\ref{thm:BT}.
	\end{thm}
	
Recall that a matrix \(A\in M_n(\mathbb R)\) is called (real) accretive  if its real part $\frac{A+A^{\top}}{2}\ge 0$. It is \emph{strictly real accretive} if $\frac{A+A^{\top}}{2}>0$. In view of the condition of Theorem \ref{thm:rankone}, we prove a sharp accretive extension of Theorem
\ref{thm:rankone}.
	
		\begin{thm}\label{thm:accretive}
		Let $A\in\mnr$ be real accretive, i.e.\ $\frac{A+A^{\top}}{2}\ge 0$. Then
		\begin{equation}\label{eq:ineq}
			\sqrt{\det A_{n-1}(1,1)\ \det A_{n-1}(2,2)}
			\ \ge\
			\left|\frac{\det A_{n-1}(1,2)+\det A_{n-1}(2,1)}{2}\right|.
		\end{equation}
	\end{thm}
	
	This paper is organized as follows. In Section~\ref{sec:Johnson}, we prove Conjecture \ref{conj:johnson}. In Section~\ref{sec:rankone}, we prove Theorem \ref{thm:rankone}. In Section~\ref{sec:accretive}, we prove Theorem \ref{thm:accretive}.
\section{Proof of Conjecture~\ref{conj:johnson}}\label{sec:Johnson}
	
	Write $\one_m\in \R^{m}$ for the all-ones column vector and note that
	$\J_m=\one_m\one_m^{\top}$ has rank $1$.
	We use $\adj(X)$ for the adjugate  of a square matrix $X$.
	The following rank-one expansion along \(\J_m\) is a classical identity for determinants, see \cite[p.~24]{HJ13}. To avoid any confusion for the reader, we provide a proof below.

	\begin{lem}\label{lem:rankone}
		Let $X$ be an $m\times m$ real matrix and let $t\in\R$. Then
		\[
		\det(X+t\J_m)=\det(X)+t\,\one_m^{\top}\adj(X)\one_m.
		\]
	\end{lem}
	
	\begin{proof}
		Since $\J_m=\one_m\one_m^{\top}$ has rank $1$, the map $t\mapsto \det(X+t\J_m)$ is a polynomial in $t$
		of degree at most $1$.
		Expanding $\det(X+t\J_m)$ by multilinearity in the columns, any term involving two columns from $t\J_m$
		vanishes because those two columns are identical. Hence only the terms with either no column from $t\J_m$
		(giving $\det X$) or exactly one column from $t\J_m$ contribute, and the coefficient of $t$ is precisely
		$\one_m^{\top}\adj(X)\one_m$.
	\end{proof}
	
	\begin{lem}\label{lem:skewfacts}
		Let $Y$ be an $m\times m$ skew-symmetric real matrix (so $Y^{\top}=-Y$). Then
		\[
		\adj(Y)^{\top}=\adj(Y^{\top})=\adj(-Y)=(-1)^{m-1}\adj(Y).
		\]
		In particular:
		\begin{enumerate}[label=\textup{(\arabic*)}, leftmargin=2.2em]
			\item if $m$ is even then $\adj(Y)$ is skew-symmetric and $\one_m^{\top}\adj(Y)\one_m=0$;
			\item if $m$ is odd then $\det(Y)=0$.
		\end{enumerate}
	\end{lem}
	
	\begin{proof}
		The displayed identity follows from $\adj(Z^{\top})=\adj(Z)^{\top}$ and
		$\adj(\lambda Z)=\lambda^{m-1}\adj(Z)$ for scalars $\lambda$.
		If $m$ is even, then $\adj(Y)^{\top}=-\adj(Y)$, so $\adj(Y)$ is skew-symmetric. In particular,
		for any skew-symmetric matrix $S$ and any vector $x$ one has $x^{\top}Sx=0$, hence
		$\one_m^{\top}\adj(Y)\one_m=0$.
		If $m$ is odd, then $\det(Y)=\det(Y^{\top})=\det(-Y)=-\det(Y)$, hence $\det(Y)=0$.
	\end{proof}
	
	\begin{proof}[Proof of Conjecture~\ref{conj:johnson}] For $n=2$, the identity \eqref{eq:johnson} can be verified by direct computation. Hence, in the remainder of the proof we assume $n\ge 3$.
		
Set $B:=A-\J_n$.
	Then $B$ is Toeplitz. By $A+A^{\top}=2 J_n$, we have
	\[
	B^{\top}=A^{\top}-\J_n=(2\J_n-A)-\J_n=-(A-\J_n)=-B,
	\]
	so $B$ is skew-symmetric.
	
	Let $m:=n-1$ and abbreviate $\one:=\one_m$ and $\J:=\J_m$.
	Define the contiguous blocks
	\[
	K:=B_m(1,1),\qquad C:=B_m(1,2).
	\]
	Then $K$ is skew-symmetric, and $B_m(2,1)=-C^{\top}$.
	Since $A=\J_n+B$, we have
	\[
	A_m(1,1)=\J+K,\qquad A_m(1,2)=\J+C,\qquad A_m(2,1)=\J-C^{\top},
	\]
	and
	\[
	\det(\J-C^{\top})=\det\bigl((\J-C^{\top})^{\top}\bigr)=\det(\J-C).
	\]
	Thus \eqref{eq:johnson} is equivalent to
	\begin{equation}\label{eq:targetB}
		\det(\J+C)+\det(\J-C)=2\det(\J+K).
	\end{equation}
	
	Introduce the scalar functional
	$s(X):=\one^{\top}\adj(X)\one$.
	By Lemma~\ref{lem:rankone} and $\adj(-X)=(-1)^{m-1}\adj(X)$, we have
\begin{align*}
	\det(\J+C)&=\det(C)+s(C),\\
	\det(\J-C)&=\det(-C)+s(-C)=(-1)^m\det(C)+(-1)^{m-1}s(C).
\end{align*}
	Hence
	\begin{equation}\label{eq:sumJC}
		\det(\J+C)+\det(\J-C)=
		\begin{cases}
			2\det(C), & m \text{ even},\\[3pt]
			2s(C), & m \text{ odd}.
		\end{cases}
	\end{equation}
	Similarly, since $K$ is skew-symmetric, Lemma~\ref{lem:skewfacts} gives
	\begin{equation}\label{eq:detJK}
		\det(\J+K)=
		\begin{cases}
			\det(K), & m \text{ even},\\[3pt]
			s(K), & m \text{ odd}.
		\end{cases}
	\end{equation}
	Thus \eqref{eq:targetB} is equivalent to
	\begin{equation}\label{eq:reduced}
		\begin{cases}
			\det(C)=\det(K), & m \text{ even},\\[3pt]
			s(C)=s(K), & m \text{ odd}.
		\end{cases}
	\end{equation}
	
	\medskip
\noindent	\textbf{Case 1: $m$ even (equivalently, $n$ odd).}
	Since $n$ is odd and $B$ is skew-symmetric, by Lemma~\ref{lem:skewfacts}\textup{(2)}, $\det(B)=0$.
	Apply Theorem~\ref{thm:dodgson} to $B$. Since $\det(B)=0$, we obtain
	\[
	0
	=
	\det B_{n-1}(1,1)\,\det B_{n-1}(2,2)
	-
	\det B_{n-1}(1,2)\,\det B_{n-1}(2,1).
	\]
	By Toeplitz structure, $B_{n-1}(1,1)=K$ and $B_{n-1}(2,2)$ is the $(n-1)\times(n-1)$ contiguous principal
	block starting at $(2,2)$, hence also equals $K$. Moreover, $B_{n-1}(1,2)=C$ and
	$B_{n-1}(2,1)=B_m(2,1)=-C^{\top}$. Therefore
	\[
	0=\det(K)^2-\det(C)\det(-C^{\top})=\det(K)^2-\det(C)^2,
	\]
	because $m$ is even and $\det(-C^{\top})=\det(C)$.
	Hence $\det(K)^2=\det(C)^2$ as an identity in the Toeplitz parameters of $B$.
	
	Now parameterize skew-symmetric Toeplitz matrices by their upper diagonals:
	write $b_1,\dots,b_{n-1}$ for the constants on the first, second, \dots, $(n-1)$st superdiagonal.
	Then $\det(K)$ and $\det(C)$ are polynomials in $b_1,\dots,b_{n-1}$.
	Since $\R[b_1,\dots,b_{n-1}]$ is an integral domain and
	\[
	(\det(K)-\det(C))(\det(K)+\det(C))=\det(K)^2-\det(C)^2=0,
	\]
	we must have $\det(K)=\det(C)$ or $\det(K)=-\det(C)$ as polynomials.
	To fix the sign, specialize to $b_1=1$ and $b_k=0$ for $k\ge 2$.
	Then $K$ is tridiagonal with $1$ on the superdiagonal and $-1$ on the subdiagonal, so its determinant
	recurrence gives $\det(K)=1$ (since $m$ is even), while $C$ is unit lower triangular (indeed $C_{ii}=1$
	and the only other nonzero entries are $C_{i,i-2}=-1$), hence $\det(C)=1$.
	Thus $\det(K)=\det(C)$ at this specialization, and therefore $\det(K)=\det(C)$ identically.
	This proves \eqref{eq:reduced} in Case 1.
	
	\medskip
\noindent	\textbf{Case 2: $m$ odd (equivalently, $n$ even).}
	Consider the one-parameter family $M(t):=B+t\J_n$.
	Applying Theorem~\ref{thm:dodgson} to $M(t)$ gives
	\begin{align}\label{eq:DJt}
		\det(M(t))\,\det\bigl(M(t)_{n-2}(2,2)\bigr)
		&=
		\det\bigl(M(t)_{n-1}(1,1)\bigr)\,\det\bigl(M(t)_{n-1}(2,2)\bigr)\nonumber\\
		&\quad-
		\det\bigl(M(t)_{n-1}(1,2)\bigr)\,\det\bigl(M(t)_{n-1}(2,1)\bigr).
	\end{align}
	
	Since $n$ is even and $B$ is skew-symmetric, Lemma~\ref{lem:skewfacts}\textup{(1)} gives
	$s(B)=\one_n^{\top}\adj(B)\one_n=0$, hence Lemma~\ref{lem:rankone} implies $\det(M(t))=\det(B)$,
	independent of $t$.
	Likewise,
	\[
	M(t)_{n-2}(2,2)=B_{n-2}(2,2)+t\J_{n-2},
	\]
	and $B_{n-2}(2,2)$ is skew-symmetric of even order $n-2$, so again Lemma~\ref{lem:skewfacts}\textup{(1)}
	and Lemma~\ref{lem:rankone} show that $\det\bigl(M(t)_{n-2}(2,2)\bigr)$ is independent of $t$.
	Therefore the left-hand side of \eqref{eq:DJt} is constant in $t$.
	
	On the right-hand side of \eqref{eq:DJt}, Toeplitz structure gives
	\[
	M(t)_{n-1}(1,1)=K+t\J,
	\qquad
	M(t)_{n-1}(2,2)=K+t\J,
	\]
	and
	\[
	M(t)_{n-1}(1,2)=C+t\J,
	\qquad
	M(t)_{n-1}(2,1)=-C^{\top}+t\J.
	\]
	Since $m$ is odd and $K$ is skew-symmetric,  by Lemma~\ref{lem:skewfacts}\textup{(2)}, we have $\det(K)=0$,
	so Lemma~\ref{lem:rankone} gives $\det(K+t\J)=t\,s(K)$.
	Similarly,
	\[
	\det(C+t\J)=\det(C)+t\,s(C),
	\qquad
	\det(-C^{\top}+t\J)=\det(-C+t\J)=-\det(C)+t\,s(C),
	\]
	since $m$ is odd implies $\det(-C)=-\det(C)$ and $\adj(-C)=\adj(C)$.
	Substituting into \eqref{eq:DJt}, the right-hand side of \eqref{eq:DJt}  becomes
	\[
	(t\,s(K))^2-(-\det(C)+t\,s(C))(\det(C)+t\,s(C))
	=
	\det(C)^2+t^2\bigl(s(K)^2-s(C)^2\bigr).
	\]
	As the left-hand side is constant in $t$, the coefficient of $t^2$ must vanish:
	\[
	s(K)^2=s(C)^2
	\]
	as an identity in the Toeplitz parameters of $B$.
	
	\medskip
	\noindent\emph{Sign determination.}
	Write $b_1,\dots,b_{n-1}$ for the Toeplitz superdiagonal parameters of $B$ (so entries of $B$ are
	$0$ or $\pm b_k$). Then $s(K)$ and $s(C)$ are polynomials in $b_1,\dots,b_{n-1}$ with \emph{integer}
	coefficients, hence
	\[
	s(K)^2-s(C)^2=0
	\quad\text{in}\quad
	\mathbb{Z}[b_1,\dots,b_{n-1}].
	\]
	Since $\mathbb{Z}[b_1,\dots,b_{n-1}]$ is an integral domain, it follows that
	either $s(K)=s(C)$ or $s(K)=-s(C)$ as polynomials over $\mathbb{Z}$.
	
	To decide the sign, evaluate over $\mathbb{Q}$ at the specialization $b_1=1$ and $b_k=0$ for $k\ge 2$.
	Then $C$ has the form $C=I-L^2$, where $L$ is the nilpotent shift matrix with $1$ on the subdiagonal.
	Since $L^{m}=0$ and $m=2\ell+1$ is odd, we have the finite Neumann series
	\[
	C^{-1}=(I-L^2)^{-1}=I+L^2+L^4+\cdots+L^{2\ell},
	\]
	so $\det(C)=1$ and $s(C)=\one^{\top}C^{-1}\one$.
	A direct computation shows
	\[
	C^{-1}\one=(1,1,2,2,\dots,\ell,\ell,\ell+1)^{\top},
	\]
	hence
	\[
	s(C)=\one^{\top}C^{-1}\one=(\ell+1)^2=\left(\frac{m+1}{2}\right)^2\neq 0
	\quad\text{in }\mathbb{Q}.
	\]
	
	For $K$, the same specialization makes $K$ tridiagonal with $1$ on the superdiagonal and $-1$ on the
	subdiagonal. Then $K$ has rank $m-1$ and $\ker(K)$ is spanned by
	\[
	u=(1,0,1,0,\dots,1)^{\top}\in\R^m.
	\]
	Since $\rank(K)=m-1$, $\adj(K)\neq 0$ and $\rank(\adj(K))=1$, and the identity $K\,\adj(K)=0$ implies
	$\adj(K)=u v^{\top}$ for some $v$. Because $m$ is odd and $K$ is skew-symmetric,
	Lemma~\ref{lem:skewfacts} gives $\adj(K)^{\top}=\adj(K)$, so $u v^{\top}$ is symmetric and hence
	$v=c\,u$ for some scalar $c$. Thus $\adj(K)=c\,u u^{\top}$.
	To determine $c$, note that
	\[
	(\adj(K))_{11}=\det\bigl(K_{m-1}(2,2)\bigr).
	\]
	The matrix $K_{m-1}(2,2)$ is an even-order tridiagonal skew-symmetric matrix with the same $\pm1$
	off-diagonal pattern, and its determinant equals $1$ by a standard recurrence. Therefore
	$(\adj(K))_{11}=1$, while $(u u^{\top})_{11}=u_1^2=1$, giving $c=1$. Hence $\adj(K)=u u^{\top}$ and
	\[
	s(K)=\one^{\top}\adj(K)\one=(\one^{\top}u)^2=\left(\frac{m+1}{2}\right)^2.
	\]
	In particular, $s(K)+s(C)=2\left(\frac{m+1}{2}\right)^2\neq 0$ in $\mathbb{Q}$ at this specialization, so
	we cannot have $s(K)=-s(C)$ as a polynomial identity. Therefore $s(K)=s(C)$ identically, proving
	\eqref{eq:reduced} in Case 2.
	
	Combining Cases 1 and 2 yields \eqref{eq:reduced}, hence \eqref{eq:targetB} holds, which proves
	Conjecture~\ref{conj:johnson}.

\end{proof}
	\begin{remark}\label{rem:basechange}
		Although the exposition above works over $\R$, the identity \eqref{eq:johnson} is in fact valid over
		every field $\Bbb F$ with $\operatorname{char}\Bbb F\neq 2$.
		Indeed, after writing $A=\J_n+B$ with $B$ skew-symmetric Toeplitz and parameters $b_1,\dots,b_{n-1}$,
		all determinants and all quantities $s(\cdot)=\one^{\top}\adj(\cdot)\one$ appearing in the proof belong to
		$\mathbb{Z}[b_1,\dots,b_{n-1}]$. The argument above determines the sign over $\mathbb{Q}$, hence yields
		the polynomial identity $s(K)=s(C)$ in $\mathbb{Z}[b_1,\dots,b_{n-1}]$; applying the coefficient map
		$\mathbb{Z}\to \Bbb F$ then gives the same identity over $\Bbb F$.
	\end{remark}
	
	\begin{remark}
		The hypothesis $A+A^{\top}=2\J_n$ is equivalent to writing $A=\J_n+B$ with $B$ skew-symmetric Toeplitz.
		The proof above shows that \eqref{eq:johnson} is ultimately a consequence of two standard tools:
		a rank-one expansion against $\J_n$ and the Desnanot--Jacobi identity.
	\end{remark}
	
	\section{Proof of Theorem~\ref{thm:rankone}}\label{sec:rankone}
	
\begin{proof}[Proof of Theorem~\ref{thm:rankone}]
	Write $K:=\frac{A-A^{\top}}{2}$, so $K^{\top}=-K$, and note that
	$A=K+\frac{\alpha}{2}ww^{\top}$.
	
	\medskip
	\noindent\textbf{Step 1: the case $w_i\neq 0$ for all $i$.}
	Let $D:=\diag(w_1,\dots,w_n)$ and define $B:=D^{-1}AD^{-1}$.
	Then
	\[
	B+B^{\top}
	=
	D^{-1}(A+A^{\top})D^{-1}
	=
	\alpha\,D^{-1}ww^{\top}D^{-1}
	=
	\alpha\,\J_n,
	\]
	so $B$ satisfies the hypothesis of Theorem~\ref{thm:BT}. Hence
	\[
	\sqrt{\det B_{n-1}(1,1)\ \det B_{n-1}(2,2)}
	=
	\left|\frac{\det B_{n-1}(1,2)+\det B_{n-1}(2,1)}{2}\right|.
	\]
	Moreover, for any row/column index sets $R,C$,
	$B_{R,C}=D_R^{-1}A_{R,C}D_C^{-1}$, and hence
	\[
	\det B_{R,C}
	=
	\left(\prod_{i\in R} w_i\right)^{-1}
	\left(\prod_{j\in C} w_j\right)^{-1}
	\det A_{R,C}.
	\]
	Applying this to the four contiguous minors gives the desired equality for $A$.
	
	\medskip
	\noindent\textbf{Step 2: the general case.}
	Let $\mathbf 1=(1,\dots,1)^{\top}$ and choose $\varepsilon>0$ sufficiently small so that $w_i+\varepsilon\neq 0$ for all $i$
	(e.g.\ $0<\varepsilon<\min\{|w_i|: w_i\neq 0\}$).
Set
	\[
	w^{(\varepsilon)}:=w+\varepsilon \mathbf 1,\qquad
	A^{(\varepsilon)}:=K+\frac{\alpha}{2}w^{(\varepsilon)}{w^{(\varepsilon)}}^{\top}.
	\]
	Then $A^{(\varepsilon)}+{A^{(\varepsilon)}}^{\top}
	=\alpha\,w^{(\varepsilon)}{w^{(\varepsilon)}}^{\top}$ and
	all components of $w^{(\varepsilon)}$ are nonzero. Hence Step~1 applies to
	$A^{(\varepsilon)}$, yielding
	\[
	\sqrt{\det A^{(\varepsilon)}_{n-1}(1,1)\ \det A^{(\varepsilon)}_{n-1}(2,2)}
	=
	\left|\frac{\det A^{(\varepsilon)}_{n-1}(1,2)+\det A^{(\varepsilon)}_{n-1}(2,1)}{2}\right|.
	\]
	Since $A^{(\varepsilon)}\to A$ entrywise as $\varepsilon\downarrow 0$ and
	each determinant is a polynomial in the entries, taking $\varepsilon\downarrow 0$
	gives the desired identity for $A$.
\end{proof}

	\section{Proof of Theorem~\ref{thm:accretive}}\label{sec:accretive}
In this section, we show that the Bayat--Teimoori identity admits a sharp extension
from the rank-one symmetric-part setting to the full cone of real accretive matrices.

First, we introduce a factorization lemma.
	\begin{lem}\label{lem:factor}
Let $A\in M_n(\mathbb R)$ be strictly accretive and write its symmetric-skew-symmetric decomposition as $A=H+N$,
where 
$H:=\frac{A+A^{\top}}{2}> 0,
N:=\frac{A-A^{\top}}{2}$, so $N^{\top}=-N$.
		Let $S:=H^{-1/2}NH^{-1/2}$, so $S^{\top}=-S$. Then
$A=H^{1/2}(I+S)H^{1/2}$
		and $(I+S)^{-1}$ is strictly accretive. In particular, $A$ is invertible and $A^{-1}$ is strictly accretive.
	\end{lem}
	
	\begin{proof}
		The factorization $A=H^{1/2}(I+S)H^{1/2}$ is immediate from the definition of $S$.
		
		Since $S^{\top}=-S$, we have $(I+S)(I-S)=I-S^2$, hence
		\[
		(I+S)^{-1}=(I-S)(I-S^2)^{-1}.
		\]
		Taking symmetric parts,
		\[
		(I+S)^{-1}+(I+S)^{-\top}
		=(I-S)(I-S^2)^{-1}+(I+S)(I-S^2)^{-1}
		=2(I-S^2)^{-1}.
		\]
		Because $S$ is real skew-symmetric, $-S^2\ge 0$ (indeed $x^{\top}(-S^2)x=\|Sx\|^2\ge 0$),
		hence $I-S^2> 0$, so $(I-S^2)^{-1}> 0$.
		Therefore $\Ree\bigl((I+S)^{-1}\bigr)=(I-S^2)^{-1}> 0$, i.e.\ $(I+S)^{-1}$ is strictly accretive.
		
		Finally,	$A^{-1}=H^{-1/2}(I+S)^{-1}H^{-1/2}$
		is strictly accretive as a congruence of a strictly accretive matrix.
	\end{proof}
	
The following two lemmas concern the sign of determinants and the accretivity of the adjugate.
	\begin{lem}\label{lem:detpos}
		If $A\in\mnr$ is accretive, then $\det A\ge 0$. If $A$ is strictly accretive, then $\det A>0$.
	\end{lem}
	
	\begin{proof}
		Assume first that $A$ is strictly accretive. By Lemma~\ref{lem:factor},
		$A=H^{1/2}(I+S)H^{1/2}$ with $H> 0$ and $S^{\top}=-S$.
		Thus $\det H>0$.
		Moreover, the eigenvalues of a real skew-symmetric matrix $S$ are purely imaginary and come in conjugate
		pairs $\{\pm i\mu_k\}$ (and possibly $0$), hence the eigenvalues of $I+S$ are $\{1\pm i\mu_k\}$ (and
		possibly $1$). Therefore the conjugate pairs contribute positive factors:
		\[
		\det(I+S)=\prod_k (1+\mu_k^2)>0,
		\]
		so $\det A=(\det H)\det(I+S)>0$.
		
		For merely accretive $A$, consider $A_\varepsilon:=A+\varepsilon I$.
		Then $\Ree A_\varepsilon=\Ree A+\varepsilon I> 0$, so $\det(A_\varepsilon)>0$ for all $\varepsilon>0$.
		By continuity of the determinant, $\det A=\lim_{\varepsilon\downarrow 0}\det(A_\varepsilon)\ge 0$.
	\end{proof}
	
	\begin{prop}\label{prop:adj-accretive}
		If $A\in\mnr$ is accretive, then $\adj(A)$ is accretive as well.
	\end{prop}
	
	\begin{proof}
		Assume first that $A$ is strictly accretive. Then by Lemma~\ref{lem:factor}, $A^{-1}$ is strictly accretive.
		By Lemma~\ref{lem:detpos}, $\det A>0$, and thus
		\[
		\adj(A)=(\det A)\,A^{-1}
		\]
		is a positive scalar multiple of a strictly accretive matrix, hence strictly accretive.
		
		For general accretive $A$, set $A_\varepsilon=A+\varepsilon I$.
		Then $\adj(A_\varepsilon)$ is accretive for each $\varepsilon>0$ by the previous paragraph.
		Since $\adj(\cdot)$ is polynomial in the entries (hence continuous) and the accretive cone
		$\{X\in\mnr:\Ree X\ge 0\}$ is closed, letting $\varepsilon\downarrow 0$ yields that $\adj(A)$ is accretive.
	\end{proof}

	\begin{proof}[Proof of Theorem~\ref{thm:accretive}]
		Consider the $2\times 2$ principal submatrix of $\adj(A)$ indexed by $\{1,n\}$:
		\[
		M:=
		\begin{pmatrix}
			(\adj(A))_{11} & (\adj(A))_{1n}\\
			(\adj(A))_{n1} & (\adj(A))_{nn}
		\end{pmatrix}.
		\]
		By Proposition~\ref{prop:adj-accretive}, $\adj(A)$ is accretive. Since taking a principal submatrix
		preserves positive semidefiniteness of the symmetric part, $M$ is accretive as well. Therefore its
		symmetric part $\Ree M=\frac{M+M^{\top}}{2}$ is a real symmetric positive semidefinite $2\times 2$ matrix,
		so its determinant is nonnegative:
		\[
		(\adj(A))_{11}(\adj(A))_{nn}
		\ \ge\
		\left(\frac{(\adj(A))_{1n}+(\adj(A))_{n1}}{2}\right)^2.
		\]
		Taking square roots and absolute values gives
		\[
		\sqrt{(\adj(A))_{11}(\adj(A))_{nn}}
		\ \ge\
		\left|\frac{(\adj(A))_{1n}+(\adj(A))_{n1}}{2}\right|.
		\]
		
		Finally, identify these adjugate entries with contiguous minors.
		By the cofactor definition of $\adj(A)$,
		\[
		(\adj(A))_{11}=\det A_{n-1}(2,2),\qquad
		(\adj(A))_{nn}=\det A_{n-1}(1,1),
		\]
		and
		\[
		(\adj(A))_{1n}=(-1)^{1+n}\det A_{n-1}(1,2),\qquad
		(\adj(A))_{n1}=(-1)^{1+n}\det A_{n-1}(2,1).
		\]
		The common factor $(-1)^{1+n}$ disappears under the absolute value, and we obtain \eqref{eq:ineq}.
	\end{proof}
We explain why the condition $\frac{A+A^*}{2}\ge 0$ is not sufficient for \eqref{eq:ineq} over $\mathbb C$.
	\begin{remark}\label{rem:complex}
		The inequality \eqref{eq:ineq} is fundamentally \emph{transpose-based}. Over $\C$, if one adopts the
		standard operator-theoretic notion of accretivity, namely $\frac{A+A^*}{2}\ge 0$, then the direct
		complex analogue of \eqref{eq:ineq} (with the same transpose-based minors on the right-hand side) can fail
		already for $n=4$.
		
		For instance, the following matrix $A\in M_4(\C)$ satisfies $\frac{A+A^*}{2}\ge 0$, but violates
		\eqref{eq:ineq}:
		\[\tiny
		A=
		\begin{pmatrix}
			9.94929343 +1.33276616i &  0.97565055 +0.87236575i & -2.50825051 +5.42561737i &  1.56748356 -7.27519505i\\
			2.97979149 -0.40625902i &  3.79277890 -0.31914688i &  0.54972864 +0.60571431i &  0.32023125 +2.04703155i\\
			-1.05662545-10.34778593i &  1.98753540 -2.33447293i &  9.41578815 -0.76975962i & -7.77317132 -1.83670880i\\
			-0.08351591 +4.49741713i &  1.36270989 -0.46531832i & -8.74961119 +1.90215917i & 16.31271805 -0.21461055i
		\end{pmatrix}.
		\]
		A direct numerical computation yields
		\[
		\sqrt{|\det A_{3}(1,1)\det A_{3}(2,2)|}\approx 168.78
		\quad<\quad
		\left|\frac{\det A_{3}(1,2)+\det A_{3}(2,1)}{2}\right|\approx 171.91.
		\]
		Thus Theorem~\ref{thm:accretive} is sharp in the real accretive setting and does not extend verbatim to
		complex accretive matrices defined via $A^*$.
	\end{remark}
	
	\section*{Acknowledgments}
	Teng Zhang is supported by the China Scholarship Council, the Young Elite Scientists Sponsorship Program
	for PhD Students (China Association for Science and Technology), and the Fundamental Research Funds for
	the Central Universities at Xi'an Jiaotong University (Grant No.~xzy022024045).

\end{document}